\documentclass[11pt]{article}

\author{}
\usepackage{amsmath,amsthm,amsfonts,amssymb}

\usepackage{graphicx}

\newcommand{\captionfonts}{\footnotesize}
\makeatletter  
\long\def\@makecaption#1#2{%
  \vskip\abovecaptionskip
  \sbox\@tempboxa{{\captionfonts #1: #2}}%
  \ifdim \wd\@tempboxa >\hsize
    {\captionfonts #1: #2\par}
  \else
    \hbox to\hsize{\hfil\box\@tempboxa\hfil}%
  \fi
  \vskip\belowcaptionskip}
\makeatother   

\newcommand{\nwc}{\newcommand}

\newtheorem{prop}{Proposition}[section]
\newtheorem{lemma}[prop]{Lemma}
\newtheorem{rem}[prop]{Remark}
\newtheorem{theorem}[prop]{Theorem}

\nwc{\R}{\mathbb R}
\nwc{\Z}{\mathbb Z}
\nwc{\N}{\mathbb N}

\newcommand{\ignore}[1]{}

\nwc{\eps}{\varepsilon}
\nwc{\re}{Re\,}

\nwc{\wto}{\rightharpoonup}

\nwc{\ds}{\displaystyle}
\newcommand {\bedis} {\begin{displaymath}}
\newcommand {\edis} {\end{displaymath}}
\newcommand{\newbeqna} {\renewcommand {\arraystretch} {2}
                        \begin {displaymath} \begin {array}{crcl}}
\newcommand{\neweqna}{\end{array} \end {displaymath}}

\newcommand{\fbeqna}{\renewcommand {\arraystretch} {1.3}
\begin {displaymath}\begin{array}{rcll}}
\newcommand{\feqna}{\end{array}\end{displaymath}}

\newcommand {\beqna} {\begin{eqnarray*}}
\newcommand {\eqna} {\end{eqnarray*}}
\newcommand {\beqn} {\begin{eqnarray}}
\newcommand {\eqn} {\end{eqnarray}}

\begin{document}
\title{Uniqueness of self-similar solutions to Smoluchowski's coagulation 
equations for kernels that are close to constant}

\author{%
B. Niethammer%
\footnote{Institute of Applied Mathematics,
University of Bonn, Endenicher Allee 60, 53115 Bonn, Germany}
and
J. J. L. Vel\'{a}zquez\footnote{Institute of Applied Mathematics,
University of Bonn, Endenicher Allee 60, 53115 Bonn, Germany}
}
\date{}

\maketitle
\begin{abstract}
We consider self-similar solutions to Smoluchowski's coagulation equation for kernels $K=K(x,y)$ that are homogeneous of degree zero and
close to constant in the sense
that
\[
 -\eps \leq K(x,y)-2 \leq \eps \Big( \Big(\frac{x}{y}\Big)^{\alpha} + \Big(\frac{y}{x}\Big)^{\alpha}\Big)
\]
for $\alpha \in [0,1)$. We prove that self-similar solutions with given mass are unique if $\eps$ is sufficiently small which is the first
such uniqueness result for kernels that are not solvable. Our proof relies on a contraction argument in a norm that measures the distance of solutions
with respect to the weak topology of measures. 
 
\end{abstract}


\section{Introduction}

In this article we consider self-similar solutions to  Smoluchowski's mean-field model for coagulation. The model applies to a system of particles in which
at any time two particles can coagulate to form a larger particle. If $\phi(\xi,t)$ denotes the number density of particles of size $\xi>0$ at time $t$, then $\phi$
satisfies the following nonlocal integral equation.
\begin{equation}\label{smolu1}
 \partial_t \phi(\xi,t) = \frac 1 2 \int_0^\xi K(\xi{-}\eta,\eta) \phi(\xi{-}\eta,t) \phi(\eta,t)\,d\eta - \phi(\xi,t) \int_0^{\infty} K(\xi,\eta) \phi(\eta,t)\,d\eta =: Q[\phi](\xi)\,.
\end{equation}
Here $K(\xi,\eta)$ denotes the so-called rate kernel, a nonnegative and symmetric function, 
that describes the rate at which particles of size $\xi$ and $\eta$ coagulate. The kernel $K$ depends  on the
microscopic details of the coagulation process and many different type of kernels can be found in the applied literature (see for example \cite{Aldous99,Drake72} and the references therein).
Most notable is Smoluchowski's kernel
\[
 K(\xi,\eta) = K_0 \Big( \xi^{1/3} + \eta^{1/3}\Big) \Big( \xi^{-1/3} + \eta^{-1/3}\Big)\,,
\]
that has been derived in Smoluchowski's original paper \cite{Smolu16} to describe coagulation in a homogeneous colloidal gold solution. The main assumptions in the derivation
are that particles are spherical, diffuse by Brownian motion when they are well-separated and coagulate quickly when two particles become close. Then $
 \Big( \xi^{1/3} + \eta^{1/3}\Big)$ is proportional to the diameter of two particles of volume $\xi$ and $\eta$ respectively, whereas $\Big( \xi^{-1/3} + \eta^{-1/3}\Big) $ is due
to Einstein's formula proportional to the diffusion constant. 

An important aspect of solutions to \eqref{smolu1} is mass conservation. Since mass is neither created or destroyed on a microscopic level, one would expect that the same is
true on the macroscopic level, that is  solutions of \eqref{smolu1} should satisfy
\begin{equation}\label{massconservation}
 \int_0^{\infty} \xi \phi(\xi,t)\,d\xi =\int_0^{\infty} \xi \phi(\xi,0)\,d\xi \qquad \mbox{ for all } t >0\,.
\end{equation}
In fact, integrating \eqref{smolu1}  and exchanging the order of integration, one finds formally \eqref{massconservation}. However, it is well-known by now
that \eqref{massconservation} is not true in general. It has first been established for the multiplicative kernel $K(x,y)=xy$, (see e.g. \cite{McLeod62a})
 and later for more general kernels
which grow faster than linearly at infinity \cite{Jeon98,EMP02}, that there is a finite time $t_*\geq 0$, that depends on the kernel and on the initial data, such that
mass is conserved up to time $t_*$ and decays afterwards. This phenomenon is known as gelation and corresponds to the creation of infinitely large clusters at the finite
time $t_*$. If, on the other hand, the kernel $K$ grows at most linearly at infinity, then mass-conservation of solutions has been established for a large range of
kernels (see e.g. \cite{Norris99,LauMisch02,LauMisch04}).

\medskip
A fundamental issue in the analysis of coagulation equations is the dynamic scaling hypothesis. It states  that for homogeneous kernels, solutions to \eqref{smolu1} converge
to a uniquely determined self-similar solution, either as time goes to infinity, or, in the case of gelation, as time approaches the gelation time. However, this issue is only
well-understood for the so-called solvable kernels, $K(x,y)=const, K(x,y)=x+y$ and $K(x,y)=xy$ for which explicit solution formulas are available. In fact, for those kernels
it has been established that there is one self-similar solution with finite mass, and convergence to this solution under some assumptions on the data has been established in 
a range of papers \cite{KreerPen94,DMR00,MePe04,LauMisch05,CMM10}. In \cite{MePe04} it was also established that in addition to self-similar
solutions with finite mass there exists  a family of self-similar solutions that have fat tails. Furthermore, in \cite{MePe04} 
the domains of attraction of all those solutions have been completely
characterized. However, the proofs of all these results rely on the use of the Laplace transform or on explicit formulas for the self-similar
solutions and cannot, at least not directly,  be extended to any other kernel.

More recently, some results on self-similar solutions to \eqref{smolu1} for kernels that are homogeneous with degree $\lambda<1$ have been established.
First, existence of self-similar profiles with finite mass for a large range of such kernels has been proved in \cite{FouLau05,EMR05} and properties of such solutions have
been investigated \cite{EsMisch06,FouLau06a,CanMisch11,NV11b}. In addition, the existence of self-similar solutions with fat tails has been
established for kernels that are bounded as $K(x,y) \leq C(x^{\lambda}+y^{\lambda})$ for $\lambda \in [0,1)$ in \cite{NV12a}. However, it has been
an open problem whether solutions with a given tail behaviour, are unique.
In this paper we present the first such result for non-solvable kernels. More precisely we prove that self-similar solutions with finite
mass are unique if the kernel $K$ is homogeneous with degree zero and
 is close to the constant kernel in the sense outlined below (see \eqref{kernel2}-\eqref{kernel0}).

To describe our result in detail we recall that self-similar solutions with finite mass to \eqref{smolu1} for kernels $K$ of homogeneity zero are given by
\begin{equation}\label{ss1}
 \phi(\xi,t) = t^{-2} f(x) \qquad \mbox{ with } x = \frac{\xi}{t}
\end{equation}
where $f$ satisfies
\begin{equation}\label{eq1}
 -xf'(x) - 2f(x) =Q[f](x)\,
\end{equation}
with
\begin{equation}\label{eq2}
 \int_0^{\infty} x f(x)\,dx =M\,.
\end{equation}
It is convenient to rewrite equation \eqref{eq1}  as 
\begin{equation}\label{eq1b}
 -\big(x^2 f(x)\big)'=x Q[f](x) = -\partial_x \int_0^x \int_{x-y}^{\infty} K(y,z) y f(z)f(y)\,dz\,dy
\end{equation}
and by integrating \eqref{eq1b} to reformulate \eqref{eq1} as   
\begin{equation}\label{eq1c}
 x^2 f(x) = \int_0^x \,dy \int_{x{-}y}^{\infty}\,dz K(y,z)y f(z)f(y)\,.
\end{equation}
We call $f$ a self-similar profile with finite mass to \eqref{smolu1} if $f \in L^1_{loc}(\R)$,  $f \geq 0$, $\int x f(x)\,dx < \infty$  and if $f$ satisfies \eqref{eq1c}
for almost all $x \in \R$. 
Notice also, that if $f$ is a solution to \eqref{eq1c}, then so is the rescaled function $g(x)=af(ax)$ for $a>0$. We can fix the parameter $a$
by fixing $M$ in \eqref{eq2}.

Our goal in this paper is to show that solutions to \eqref{eq1c} and \eqref{eq2} are unique if the kernel $K$ is close to the constant one. More precisely we
make the following
assumptions on the kernel:

We assume for the kernel $K\colon (0,\infty)^2 \to [0,\infty)$ 
that \begin{equation}\label{kernel1}
 K \mbox{ is homogeneous of degree zero, that is } K(\lambda x,\lambda y) = K(x,y) \; \mbox{ for all } x,y,\lambda>0\,.
\end{equation}
Furthermore we assume that there exists $\eps>0$ and $\alpha \in [0,1)$ such that
\begin{equation}\label{kernel2}
W(x,y):=K(x,y)-2  \geq -\eps\,,\qquad \mbox{ for all } x,y>0\,,
\end{equation}
\begin{equation}\label{kernel3}
 W(x,y) \leq \eps \Big( \Big(\frac{x}{y}\Big)^{\alpha} + \Big (\frac{y}{x}\Big)^{\alpha}\Big)\qquad \mbox{ for all } x,y>0
\end{equation}
and that  $K$ is differentiable with
\begin{equation}\label{kernel0}
 \Big | \frac{\partial}{\partial x} K(x,y) \Big| \leq  \frac{C\eps}{x} \Big( \Big(\frac{x}{y}\Big)^{\alpha} + \Big (\frac{y}{x}\Big)^{\alpha}\Big)\qquad \mbox{ for all } x,y>0\,.
\end{equation}
The last assumption could be weakened in the sense that it would suffice that a H\"older norm of $K$ is small locally with a certain blow-up rate as $x,y \to 0$. Assumption \eqref{kernel0}
is just somewhat easier to formulate and it is also satisfied (up to the smallness assumption) by kernels one typically encounters in applications.

\begin{theorem}\label{T.uniqueness}
Assume that $K$ satisfies the assumptions \eqref{kernel1}-\eqref{kernel0} and let $f_1$ and $f_2$ be two self-similar profiles  that satisfy \eqref{eq2}. 
Then, if  $\eps$ is sufficiently small, we have $f_1=f_2$. 
\end{theorem}

The key ingredients of our proof are the following. In Section \ref{S.apriori} we collect several a priori estimates. 
First, we need certain regularity
of the solutions as $x\to 0$ and it is for those  estimates that we need a uniform lower bound on the kernel. In fact, it is known that for kernels that are not uniformly bounded
away from zero (e.g. the diagonal kernel) solutions have less regularity than what we need for our proof. In order to derive these results and also for the contraction argument
in the uniqueness proof we consider, as in \cite{MePe04} for the solvable kernels, the desingularized Laplace transform of $f$, for which we can derive an approximate differential
equation (see Lemma \ref{L.Qequation}).
Another key estimate is  that any self-similar solution with finite mass decays exponentially as $x \to \infty$ (cf. Lemma \ref{L.ublargex} and Lemma \ref{C.ublargex}).
 This result and more detailed estimates for the behaviour for large
$x$ is contained in  \cite{NV13a}. For completeness we present the proof of the upper bound that is needed here in the Appendix.
 In Lemma \ref{L.qclose} we show that the self-similar solution is close to the one for the constant kernel in the sense that their Laplace transforms are 
close. 
The contraction argument that gives uniqueness is contained in Section \ref{S.uniqueness}. Again, the key idea is to consider a suitable norm (cf. \eqref{normdef}) that is
a weighted norm of the desingularized Laplace transform and hence measures the distance of solutions in the weak topology.

\bigskip
For the following it is convenient to use the normalization $M=1$ in \eqref{eq2}, such that the self-similar solution for $K=2$ is $f(x)=e^{-x}$.

\section{A priori estimates }\label{S.apriori}

\subsection{Properties of the Laplace transform}

For the following we use what is sometimes called the desingularized Laplace transform of $f$, given by
\begin{equation}\label{desinglaplaceg}
 Q(q)= \int_0^{\infty} \big(1-e^{-qx}\big) f(x)\,dx\,.
\end{equation}
The function $Q$ is defined for all $q \geq 0$ and due to \eqref{eq2} we have $Q(0)=0$. 
Normalizing the mass $M=1$ also implies that $Q'(0)=1$. We will see later, see  Lemma \ref{L.ublargex}, that the function $Q$ is defined on $(-\delta,\infty)$ for some
$\delta >0$. 

\bigskip
For the following we define
\begin{equation}\label{mdef}
 {\cal M}(f,f)(q) =\frac 1 2 \int_0^{\infty} \int_0^{\infty} W(x,y) f(x)f(y) \big(1-e^{-qx}\big) \big( 1- e^{-qy}\big)\,dx\,dy\,.
\end{equation}
We first need to show via some a-priori estimates that ${\cal M}(f,f)(0)=0$.
\begin{lemma}\label{L.mproperty}
 If $K(x,y) \geq c_0>0$ and if $f$ is a solution to \eqref{eq1c} and \eqref{eq2} then
\[
 \lim_{q \to 0} {\cal M}(f,f)(q)=0\,.
\]
\end{lemma}
\begin{proof}
 We first notice that the proof of Lemma 2.1 in \cite{NV11b} applies without any change to conclude that
\begin{equation}\label{extra1}
 \sup_{R>0} \frac{1}{R} \int_{R/2}^R x f(x)\,dx \leq C\,.
\end{equation}
Then, by a dyadic argument, we conclude with \eqref{extra1} that
\begin{equation}\label{ubsmallx}
  \begin{split}
  \int_0^1y^{1-\alpha} f(y)\,dy & \leq \sum_{n=0}^{\infty} \int_{2^{-(n{+}1)}}^{2^{-n}} y^{1{-}\alpha} f(y)\,dy\\
& \leq \sum_{n=0}^{\infty} 2^{n\alpha} \int_{2^{-(n{+}1)}}^{2^{-n}} y f(y)\,dy\\
& \leq \sum_{n=0}^{\infty} 2^{n(\alpha-1)} \leq C\,.
 \end{split}
\end{equation}
As a consequence, we can estimate
\[
 \begin{split}
  \int_0^{1/2q}\int_0^{1/2q}& |W(x,y)| f(x)f(y)\big(1-e^{-qx}\big)\big(1-e^{-qy}\big) \,dx\,dy\\
 & \leq 
C q^2 \int_0^{1/2q}\int_0^{1/2q} \Big( \Big(\frac{x}{y}\Big)^{\alpha} + \Big(\frac{y}{x}\Big)^{\alpha}\Big) yx f(x)f(y)\,dx\,dy\\
& \leq C q^{2-2\alpha} \to 0 \qquad \mbox{ as } q \to 0\,.
 \end{split}
\]
Furthermore, using \eqref{eq2}, we have
\[
 \begin{split}
  \int_{1/2q}^{\infty} \int_{1/2q}^{\infty}& |W(x,y)| f(x)f(y) \big(1-e^{-qx}\big)\big(1-e^{-qy}\big) \,dx\,dy \\
& \leq 
C \int_{1/2q}^{\infty} \int_{1/2q}^{\infty} x^{\alpha} y^{\alpha} f(x) f(y)\,dx\,dy\\
& \leq C q^{2-2\alpha} \to 0 \qquad \mbox{ as } q \to 0
 \end{split}
\]
and we can similarly conclude that the term $\int_0^{1/2q} \int_{1/2q}^{\infty} \,dy\,dx ....$ converges to zero as $q \to 0$, which proves the claim. 
\end{proof}

\bigskip
To obtain further estimates we  derive a differential equation for $Q$.
\begin{lemma}\label{L.Qequation}
 The function $Q$ satisfies for all $q$ with $Q(q)<\infty$ that 
\begin{equation}\label{qequation}
 -q Q'(q) = Q^2-Q + {\cal M}(f,f)(q)\,.
\end{equation}
\end{lemma}
\begin{proof}
Multiplying \eqref{eq1c} by $e^{-qx}$ and integrating we find, after changing the order of integration, that 
\[
 \begin{split}
-Q^{''}(q) &=  \int_0^{\infty} x^2 f(x) e^{-qx}\,dx \\
&=\int_0^{\infty} \int_0^{\infty} K(y,z) y f(y) f(z) \int_{y}^{y+z} e^{-qx}\,dx\,dy\,dz\\
&= \int_0^{\infty} \int_0^{\infty} K(y,z) y f(y) f(z) \frac{1}{q} e^{-qy} \Big(1 - e^{-qz}\Big)\,dy\,dz\\
&= \frac{2}{q} Q'(q) Q(q) + \frac{1}{q}  {\cal M}(f,f)'(q)
 \end{split}
\]
and as a consequence we find
\[
 - \big( q Q'\big)' = \big(Q^2\big)' - Q' + \big( {\cal M}(f,f)\big)'\,.
\]
By definition, we have $Q(0)=0$ and Lemma \ref{L.mproperty} implies that ${\cal M}(f,f)(0)=0$. Hence, integrating the previous identity we deduce the claim.
\end{proof}
In the following we denote by $\bar Q$ the desingularized Laplace transform for the case $K=2$, that is
\begin{equation}\label{Qbardef}
 \bar Q(q) = \int_0^{\infty} e^{-x}\big(1-e^{-qx}\big)\,dx = 1-\frac{1}{1+q}=\frac{q}{1+q}\,.
\end{equation}

In the following Lemma we derive some  a-priori estimates for $Q$ and ${\cal M}$ that are essential for our analysis
and follow rather easily from the lower bound on $K$.

\begin{lemma}\label{L.apriori}
If  $K(x,y) \geq c_0>0$ for all $x,y>0$, then the following estimates hold.
\begin{align}
 \lim_{q \to \infty} Q(q)&<\infty \label{qinfty}\qquad \mbox{ and hence } \int_0^{\infty} f(x)\,dx < \infty\,,\\
\sup_{q >0} |q Q'(q)|&\leq C\,,\label{qprimebound}\\
\int_0^{\infty}\int_0^{\infty} K(x,y) f(x) f(y)\,dx\,dy&< \infty\,, \label{kintegralbound}\\
\lim_{q \to \infty} {\cal M}(f,f)(q)& <\infty\label{Mlimit}.
\end{align}
\end{lemma}
\begin{proof}
With the assumption on $K$ we can deduce from \eqref{qequation}, written with $K$ instead of $W$, that
\[
 -q Q'(q) = -Q + \int_0^{\infty} \int_0^{\infty} K(x,y) f(x)f(y) \big(1{-}e^{-qx}\big)\big(1{-}e^{-qy}\big)dxdy \geq - Q + c_0 Q^2.
\]
Hence, by comparing with the solution of the corresponding ODE, the function $Q$ is uniformly bounded. Since $Q$ is increasing, statement \eqref{qinfty} follows.

Next, we have
\[
 Q'(q) = \frac{1}{q} \int_0^{\infty} xq e^{-xq} f(x)\,dx \leq \frac{C}{q} \int f(x)\,dx \,,
\]
which together with \eqref{qinfty} establishes \eqref{qprimebound}.

Then it follows from the equation for $Q$ that 
\[
 \int_0^{\infty} \int_0^{\infty} K(x,y) f(x)f(y) \big(1-e^{-qx}\big)\big(1-e^{-qy}\big)\,dx\,dy\leq C
\]
and by monotone convergence we find \eqref{kintegralbound} in the limit $q \to \infty$. Denoting this limit by $J$ we finally get that
\begin{align*}
 {\cal M}(f,f)(q)& = \int_0^{\infty} \int_0^{\infty} W(x,y) f(x)f(y) \big(1-e^{-qx}\big)\big(1-e^{-qy}\big)\,dx\,dy\\
&=\int_0^{\infty} \int_0^{\infty} K(x,y) f(x)f(y) \big(1-e^{-qx}\big)\big(1-e^{-qy}\big)\,dx\,dy - Q(q)^2\\
& \to J - Q(\infty)^2\,
\end{align*}
which proves \eqref{Mlimit}.
\end{proof}

\subsection{Regularity near zero}

In the following Lemma we prove a certain regularity for $f$ as $x \to 0$. As already mentioned in the introduction, this result relies on a uniform lower bound on the kernel.
In fact, for the diagonal kernel, the corresponding result is known not to be true, since solutions behave as $f(x) \sim \frac{C}{x}$ as $x \to 0$ and thus 
\eqref{regularity} and \eqref{negativemoment}
do not hold. 

\begin{lemma}\label{L.regularity}
Given $\eta>0$ there exists  $\rho_0>0$ such that for sufficiently small $\eps$
\begin{equation}\label{regularity}
\int_{\rho}^{2\rho} f(x)\,dx \leq C\rho^{1-\eta} \qquad \mbox{ for all } \rho \in (0,\rho_0]\,.
\end{equation}
As a consequence we obtain
\begin{equation}\label{negativemoment}
 \int_0^1 \frac{f(x)}{x^{\alpha}}\,dx \leq C_{\alpha}\,.
\end{equation}
\end{lemma}
\begin{proof}
 We have seen in Lemma \ref{L.apriori} that $L:=\lim_{q \to \infty} {\cal M}(f,f)(q)$ exists. Furthermore, we deduce from 
\eqref{qequation} that 
\begin{equation}\label{qlrelation}
 Q^2(\infty)-Q(\infty)+L=0\,.
\end{equation}
Using \eqref{kernel2}, \eqref{qequation} and \eqref{qlrelation} we can derive the following differential inequality for the positive function $Z(q):=Q(\infty)-Q(q)$:
\begin{align*}
 qZ'&= Q^2-Q+ {\cal M}(f,f)\\
&= Z^2 + \big (1-2Q(\infty)\big) Z -L + {\cal M}(f,f)\\
&=Z^2 + \big (1-2Q(\infty)\big) Z + \tfrac 1 2 \int_0^{\infty} \int_0^{\infty} W(x,y) f(x) f(y) \Big( \big(1-e^{-qx}\big)\big(1{-}e^{-qy}\big) {-}1\Big)dxdy  \\
& \leq Z^2 + \big (1-2Q(\infty)\big) Z -C\eps \Big( Q(q)^2-Q(\infty)^2\Big)\\
&\leq (1-C\eps) Z^2 + \big(1-2(1-C\eps) Q(\infty)\big) Z\,.
\end{align*}
We know that given $\delta>0$ we have $Q(\infty) \in (1-\delta,1+\delta)$ if $\eps$ is sufficiently small. Hence $Z(q) \leq \delta $ for $q \geq \hat q$ where $\hat q$ is 
sufficiently large. As a consequence we obtain 
\[
 qZ'(q) \leq (\eta-1) Z \qquad \mbox{ with } \eta = (1-C\eps)\delta + 2C(\eps+\delta)) \,,
\]
which implies
\begin{equation}\label{Zbound}
 Z(q) \leq Z(\hat q) \Big( \frac{\hat q}{q}\Big)^{1{-}\eta} 
\end{equation}
and thus
\begin{equation}\label{Zbound1}
 Z(q) = \int_0^{\infty} f(x) e^{-qx}\,dx \leq \frac{C}{q^{1{-}\eta}}\,\qquad \mbox{ for } q \geq \hat q\,.
\end{equation}
Choosing $q=\frac{1}{\rho}$, the estimate \eqref{regularity} follows.  
To obtain \eqref{negativemoment} we use  a dyadic argument. More precisely, we estimate for $\eta < 1-\alpha$ that
\[
 \int_0^1 \frac{f(x)}{x^{\alpha}}\,dx  = \sum_{n=0}^{\infty} \int_{2^{-(n+1)}}^{2^{-n}} \frac{f(x)}{x^{\alpha}}\,dx \\
\leq C \sum_{n=0}^{\infty} 2^{-(1{-}\eta {-}\alpha)n }\leq C\,.
\]

\end{proof}

\subsection{Exponential decay} 

A key result for our analysis is the following decay estimate. If $f$ is a  solution of \eqref{eq1c} and \eqref{eq2} then it  decays exponentially fast.
This fact as well as stronger results can be proved for a much larger class of kernels than considered in this paper (see  \cite{NV13a}). For the convenience
of the reader we present the proof of Lemma \ref{L.ublargex} in the appendix.

\begin{lemma}
 \label{L.ublargex}
There exist constants $C,a>0$ such that any solution of \eqref{eq1c}, \eqref{eq2} satisfies
\[
 f(x) \leq C e^{-ax} \qquad \mbox{ for all } x\geq 1\,.
\]
\end{lemma}

\begin{rem}
 Due to the invariance of \eqref{eq1c} under rescaling, we can obtain that $f(x) \leq Ce^{-x}$, but have to give up \eqref{eq2} instead.
\end{rem}

As a consequence of Lemma \ref{L.ublargex} one also  obtains the following result.
\begin{lemma}\label{C.ublargex}
 Let $f(x)$ be a solution to \eqref{eq1c} and \eqref{eq2} such that $\int_0^{\infty} f(x) e^{ax}\,dx < \infty$ for $a>0$. Then there exists $b>0$ such that
$f(x)e^{ax} \leq C e^{-bx}$ for all $x \geq 1$.
\end{lemma}
\begin{proof}
The statement follows from the observation that the function $g(x)=f(x)e^{-ax}$ satisfies the inequality
\[
 x^2 g(x) = \int_0^x \,dy \int_{x-y}^{\infty}  \,dz K(y,z) e^{a(x-(y+z))} y g(y)g(z) \leq \int_0^{x} \,dy \int_{x-y}^{\infty} \,dz K(y,z) y g(y)g(z)\,,
\]
which is sufficient to apply the proof of Lemma \ref{L.ublargex} to $g(x)$. 
\end{proof}

\subsection{The solution is close to the one for the constant kernel}

Our next Lemma shows  that $Q$ is close  to $\bar Q$ for small $\eps$ as long as we stay away from the singularity of $\bar Q$, that is $q=-1$.

\begin{lemma}\label{L.qclose}
 Given $\delta>0$ and $\nu>0$, we have for sufficiently small  $\eps>0$  that
\[
 \sup_{q > -1+\nu} |(Q-\bar Q)(q)| \leq \delta\,.
\]
\end{lemma}

\begin{proof}
We denote $G(q):=Q(q)-\bar Q(q)$ such that $G$ satisfies the equation
\begin{equation}\label{Gequation}
 -qG'(q)= \big(2\bar Q-1\big) G + G^2 +{\cal M}(f,f)(q)
\end{equation}
and $G(0)=0$ as well as due to  our normalization $G'(0)=0$.
Integrating \eqref{Gequation} we find
\begin{equation}\label{Grepresentation}
 \frac{G(q)}{q} = \frac{G(q_0)}{q_0} - \frac{1}{(1+q)^2} \int_{q_0}^q (1+r)^2 \frac{G^2(r)}{r^2}\,dr + \frac{1}{(1+q)^2} \int_{q_0}^q (1+r)^2 \frac{{\cal M}(f,f)(r)}{r^2}\,dr\,.
\end{equation}
We first consider $q_0=0$ and recall that by our assumptions $\lim_{q \to 0} \frac{G(q)}{q}=0$. For  $\rho\geq 0$ define  
\[
 \|G\|_{\rho}:=\sup_{|q|\leq \rho} \Big| \frac{G(q)}{q}\Big|\,,
\]
such that $\|G\|_{0} =0$. From Lemma \ref{L.ublargex} we know that there exists $\eta>0$ such that $Q$ and hence $G$ are defined for all  $q\in [-\eta,\infty)$.
By linearizing $1-e^{-qx}$, which is possible due to Lemma \ref{C.ublargex},  we have the estimate
\[
|{\cal M}(f,f)(q)|\leq C_{\eta}\eps q^2 \int_0^{\infty} \int_0^{\infty} \Big( \Big( \frac{x}{y}\Big)^{\alpha} + \Big ( \frac{y}{x}\Big)^{\alpha}\Big) xy f(x) f(y)\,dx\,dx 
\leq C_{\eta} \eps q^2
\]
for $-\eta <q<\infty $. Now let $\rho \in (0,\eta]$ be such that $\|G\|_{\rho}  \leq \frac{1}{2}$\,. Then, we obtain from 
\eqref{Grepresentation}
\[
 \frac{G(q)}{q} \leq \frac{1}{2} \int_0^q \frac{G(r)}{r}\,dr + C_{\eta} \eps \rho
\]
and Gronwall's inequality implies
\[
 \frac{G(q)}{q} \leq C_{\eta} \eps \rho\,,
\]
which implies that we can choose $\rho = \eta$ and have the desired estimate in $[-\eta,\eta]$. 

We are now going to derive the  estimate in $[-1+\nu,-\eta]$. To that aim observe that ${\cal M}$ can be estimated,
recalling \eqref{negativemoment}, by
\begin{equation}\label{Mestimate}
\begin{split}
 {\cal M}(f,f)(q)&\leq C \eps \int_0^{\infty} x^{\alpha} f(x) \big( 1-e^{-qx}\big)\,dx \int_0^{\infty} y^{-\alpha} f(y) \big(1-e^{-qy}\big)\,dy\\
& \leq C \eps \big( 1+|Q'|\big)^{\alpha} |Q|^{1-\alpha} \big(1+|Q|\big) \\
& \leq C \eps \big( 1 + |Q'| + |Q|^{\frac{2-\alpha}{1-\alpha}}\big)\,.
\end{split}
\end{equation}
We know that $|\bar Q(q)|\leq C_{\nu}$ for $q \in [-1+\nu,-\eta]$. We consider now an interval $[-\rho,-\eta]$ such that $|Q(q)| \leq 2|\bar Q(q)|\leq 2C_{\nu}$. 

We know from Lemma \ref{C.ublargex} that  $Q(q)$ is defined on a larger interval, if $Q(q)$ remains bounded.
Then \eqref{qequation} and \eqref{Mestimate} imply that in $[-\rho,-\eta]$ the function $Q$ satisfies an equation of the form
\[
 -q \Big(1+\frac{a(q)}{q}\Big) Q' = Q^2 - Q + b(q)
\]
with 
\[
 |a(q)|\leq C \eps \qquad \mbox{ and } \qquad |b(q)|\leq C \eps\,
\]
and by linearization
\begin{equation}\label{eqdiff}
 -qQ'= Q^2-Q + \sigma \qquad \mbox{ with } |\sigma(q)|\leq C \eps. 
\end{equation}
Then $G$ solves
\[
 -q G'= \big( 2 \bar Q -1 \big) G + G^2 + \sigma\,, \qquad |G(-\eta)| \leq \delta_1
\]
where $\delta_1$ can be made arbitrarily small if $\eps$ is small. 
We can then use the representation formula \eqref{Grepresentation} for $G$ and Gronwall's inequality to conclude that
\[
 |G(q)| \leq C(\delta_1 + \eps) \qquad \mbox{ for all } q \in [-\rho,-\eta]
\]
and this in turn implies that we can take $\rho=-1+\nu$ and we have the desired estimate in $[-1+\nu,\-\eta]$. 

The corresponding estimate in $[\eta,\infty)$ follows similarly, using that due to \eqref{qinfty} we have a uniform bound on $Q$ and thus we also have \eqref{eqdiff}.
\end{proof}

Our next Lemma shows that  $Q$ blows up at
a point $q^*$ that is close to $-1$ and it blows up with the same rate as $\bar Q$.

\begin{lemma}\label{L.singularity}
Given $\delta>0$ there exists $\eps>0$ such that there exists $q^*$ with
$|q^*+1|\leq \delta$ and $\lim_{q \to q^*} |Q(q)|=\infty$.

Furthermore there exists $r>0$ such that
\begin{equation}\label{singularity}
 \big| (q-q^*)Q(q)+1\big| \leq \delta \qquad \mbox{ for all } q \in (q^*,q^*+r)\,.
\end{equation}
 \end{lemma}

\begin{proof}
From the previous Lemma we know that $q^* \leq -1+\nu$ where $\nu$ can be made arbitrarily small with $\eps$. 
To obtain a lower bound on $q^*$  we return to \eqref{Mestimate}
and derive
\begin{align}
 {\cal M}(f,f)(q) &\leq C \eps \big(1+|Q'|\big)^{\alpha} |Q|^{1-\alpha} |V|\nonumber\\
& \leq C \eps \big( 1 + |Q'| + |Q| |V|^{\frac{1}{1{-}\alpha}}\big) \label{Mestimate1}
\end{align}
with
\[
 V(q)=\int_0^{\infty} x^{-\alpha} f(x) \big(1-e^{-qx}\big)\,.
\]
We find that $V$ satisfies
\begin{equation}\label{Vhoelder}
 |V(q)-V(\hat q)| \leq C \big(1+|Q(\hat q)|\big) |q-\hat q|^{\alpha}
\end{equation}
Indeed, this follows from 
\begin{align}
\big| V(q)-V(\hat q)\big| &= \Big|\int_0^{\infty} x^{-\alpha} \big(e^{-qx} - e^{-\hat qx}\big) f(x)\,dx \Big|\nonumber \\
& \leq \int_0^{\infty}  e^{-\hat q x} |1-e^{-x(q-\hat q)}| \frac{1}{(x|q-\hat q|)^{\alpha}} |q-\hat q|^{\alpha} f(x)\,dx\nonumber \\
& \leq C |q-\hat q|^{\alpha}\int_0^{\infty} e^{-\hat q x} f(x)\,dx \label{Mestimate2}\\
& \leq C |q-\hat q |^{\alpha} \big( 1 + |Q(\hat q)|\big)\,.\nonumber
\end{align}

Given $\eta>0$ we now choose $\nu$ in Lemma \ref{L.qclose} such that with $q_0=-1+\nu$
\[
 |Q(q_0)| \leq \eta |Q(q_0)|^2 \qquad \mbox{ and } \qquad |\bar Q(q_0)| \leq \eta |\bar Q(q_0)|^2\,.
\]
Then we define a decreasing sequence $q_n$ in the following way:
\begin{equation}\label{qndef}
 q_{n+1}=q_n - \frac{1}{4|Q(q_n)|}\,.
\end{equation}
We are going to show by induction that 
\begin{align}
 |V(q_{n{+}1})|^{\frac{1}{1-\alpha}} & \leq C |Q(q_n)| \,,\label{vestimate}\\
\frac{1}{2}|Q(q_n)| & \leq |Q(q_{n+1})| \leq 2 |Q(q_n)|\,, \label{Q1}\\
 |Q(q_{n{+}1})| &\geq \frac 7 6 |Q(q_n)| \,. \label{Q2}
\end{align}
In fact, it follows from  \eqref{Mestimate2} and \eqref{qndef}  that
\begin{align*}
 |V(q_{n{+}1})|^{\frac{1}{1{-}\alpha}} & \leq |V(q_n)|^{\frac{1}{1{-}\alpha} } + |Q(q_n)|^{\frac{1}{1{-}\alpha}} |q_{n{+}1}-q_n|^{\frac{\alpha}{1{-}\alpha}}\\
& \leq C_{\nu}  + |Q(q_n)|  \leq C |Q(q_n)|\,.
\end{align*}
Inserting \eqref{Mestimate1} into  \eqref{Mestimate2} we obtain for $q \in [q_{n{+}1},q_n]$, taking also into account that $|Q(q)|$ is increasing for decreasing $q$, that 
\begin{equation}\label{Mestimate3}
 {\cal M}(f,f)(q) \leq C \eps \big( 1 + |Q'(q)| + |Q(q)|^2\big)
\end{equation}
As a consequence, we obtain that $Q$ satisfies for $q \in [q_{n{+}1},q_n]$ 
\[
 -Q'(q) = \big( 1 + a(q)\big) Q^2 \qquad \mbox{ with } |a(q)| \leq C \big( \eps + \eta + \nu\big)\,.
\]
Integrating this equation, we find
\begin{align}
 Q(q) &= \frac{Q(q_n)}{1- Q(q_n) \big( 1 + O(\eps+\eta+\nu)\big) (q-q_n)}\nonumber\\
& = \frac{Q(q_n)}{1- |Q(q_n)| \big( 1 + O(\eps+\eta+\nu)\big) |q-q_n|} \label{Qsolution}
\end{align}
and in particular, due to the monotonicity of $Q$ and the definition of the sequence $\{q_n\}$ in \eqref{qndef}, we 
deduce \eqref{Q1} and \eqref{Q2}.

Then
\begin{align*}
 q_{n+1} & = q_0 - \frac{1}{4} \Big( \frac{1}{|Q(q_0)|} + \cdots + \frac{1}{|Q(q_n)|}\Big)\\
& \geq q_0 - \frac{1}{4|Q(q_0)|} \Big( 1 + \frac{6}{7} + \Big( \frac{6}{7}\Big)^2 \cdots \Big) \to q_0 - \frac{7}{4|Q(q_0)|} \,. 
\end{align*}
As a consequence of this and \eqref{Q2}, we obtain that $Q$ blows up at a point $q^* \geq q_0-\frac{7}{4|Q(q_0)|}$. 

It remains to prove \eqref{singularity}.
We return to \eqref{Qsolution} to obtain
\[
 Q(q_{n+1}) = \frac{Q(q_n)}{ 1- \frac{1}{4} (1+ O(\eps + \eta +\nu ) )} = \frac{4}{3} Q(q_n) (1+ O(\eps+ \eta+\nu) )\,.
\]
Iterating this argument we find
\begin{align*}
 \Big( \frac 4 3 \Big)^{k-(n+1)} &|Q(q_{n+k})| \big( 1- O(\eps+\eta+\nu)\big)
\leq |Q(q_k)| \\
&\leq \Big( \frac 4 3 \Big)^{k-(n+1)} |Q(q_{n+k})| \big( 1+O(\eps+ \eta+\nu) \big)\,.
\end{align*}
As a consequence
\[
 q_{n+1}-q^* = \frac{1}{4} \sum_{k \geq n+1} \frac{1}{|Q(q_k)|} \geq \frac{1}{4 |Q(q_{n+1})|} \sum_{l=0}^{\infty} \Big( \frac 3 4 \Big)^l \frac{1}{1+C(\eps+\eta
+\nu)^l}
\]
and
\[
q_{n+1}-q^* = \frac{1}{4} \sum_{k \geq n+1} \frac{1}{|Q(q_k)|} \leq \frac{1}{4 |Q(q_{n+1})|} \sum_{l=0}^{\infty} \Big( \frac 3 4 \Big)^l \frac{1}{1-C(\eps 
 + \eta+\nu)^l}\,.
\]
Hence
\[
 \big(1- C(\eps+ \eta+\nu)\big) \frac{1}{|Q(q_{n+1})|} \leq q_{n+1}-q^* \leq \big(1+ C(\eps +  \eta+\nu)\big) \frac{1}{|Q(q_{n+1})|}\,.
\]
Since
\[
 |Q(q)| = \frac{|Q(q_n)|}{1+|Q(q_n)|(1+O(\eps+\eta+\nu))(q-q_n)} = \frac{1}{q-q_n} \Big( 1 + O(\eps + \eta+\nu)\Big) 
\]
we also find
\[
 |Q(q)| = \frac{1}{q-q^*}  \Big( 1 + O(\eps  + \eta+\nu)\Big) 
\]
and the proof of \eqref{singularity} is finished.
\end{proof}

\section{Uniqueness proof}
\label{S.uniqueness}

From now on we rescale the solution such that the singularity of its desingularized Laplace transform $Q$ is at $q=-1$. We denote the corresponding functions
again by $f$ and $Q$ respectively.

Since all the transforms are defined on the interval $(-1,\infty)$ we can define the following norm, that is particularly suited for our uniqueness proof:
\begin{equation}\label{normdef}
 \|Q\|:= \sup_{q>-1} \frac{1+q}{|q|}|Q(q)|\,.
\end{equation}

As a corollary of Lemmas \ref{L.qclose} and \ref{L.singularity} we obtain the following. 

\begin{lemma}\label{L.qgloballyclose}
 Given $\delta>0$ there exists $\eps>0$ such that 
\begin{equation}\label{smallness}
 \| Q-\bar Q\|\leq \delta\,.
\end{equation}
\end{lemma}

\subsection{The representation formula}

Our next goal is to derive a representation formula for $U:=Q-\bar Q$.  Then $U$ satisfies the equation
\begin{equation}\label{uequation}
 -q  U'(q) = \big( 2 \bar Q -1\big)U  +U^2 + {\cal M}(f,f)(q)\,
\end{equation}
and $U = o\big( \frac{1}{1+q}\big)$ as $q \to -1$.

\begin{lemma}\label{L.representation}
The solution to  \eqref{uequation} can be represented as  
\begin{equation}\label{urepresentation}
 U(q) = - \frac{q}{(1+q)^2} \int_{-1}^q  \frac{(1+s)^2}{s^2} \int_0^s  \,\psi(\eta)\,d\eta\,ds \qquad \mbox{ with } \psi= U^2 + {\cal M}(f,f)\,.
\end{equation}
 Furthermore, if $U_1$ and $U_2$ are two such solutions, then
\begin{equation}\label{udifference}
\begin{split}
 U_1(q)-U_2(q)& = - \frac{q}{(1+q)^2} \int_{-1}^q  \frac{(1+s)^2}{s^2} \Big( U_1(s)^2 - U_2(s)^2\Big)\,ds 
\\ & 
- \frac{q}{(1+q)^2} \int_{-1}^q  \frac{(1+s)^2}{s^2} \Big( {\cal M}(f_1,f_1)(s) - {\cal M}(f_2,f_2)(s)\Big)\,ds \,.
\end{split}
\end{equation}
\end{lemma}
\begin{proof}
Integrating the equation
\[
 -qU'(q)= (2\bar Q-1)U + \psi= \Big(1-\frac{2}{1+q}\Big) U +\psi
\]
gives
\[
 \Big( \frac{(1+q)^2}{q} U\Big)' = - \Big( \frac{1+q}{q}\Big)^2 \psi
\]
and thus \eqref{urepresentation} follows.
\end{proof}

\subsection{The contraction argument}

\begin{prop}
Let $U_1$ and $U_2$ be two solutions of \eqref{uequation} as in Lemma \ref{L.representation} then we have $U_1=U_2$ if $\eps>0$ is sufficiently small.
\end{prop}
\begin{proof}
We deduce from  \eqref{udifference} that
\begin{equation}\label{prop1}
 \begin{split}
\|U_1-U_2\|&\leq \sup_{q>-1}  \frac{1}{q{+}1}  \int_{-1}^q  \frac{(1+s)^2}{s^2} \Big | U_1(s)^2 - U_2(s)^2\Big| \,ds\\
& \quad + \sup_{q>-1}  \frac{1}{q{+}1}\Big| \int_{-1}^q  \frac{(1+s)^2}{s^2} \Big( {\cal M}(f_1,f_1)(s) - {\cal M}(f_2,f_2)(s)\Big)\,ds \Big|\\
& =: (I) + (II)\,.
 \end{split}
\end{equation}
The first term is easy to estimate. In fact, using \eqref{smallness}, we find for sufficiently small $\eps$ that 
\begin{equation}\label{prop2}
 \begin{split}
|(I)|& \leq \sup_{q>-1} \frac{1}{(1{+}q)} \int_{-1}^q \Big( \|U_1\|+\|U_2\|\Big)  \|U_1-U_2\|  \,ds \\
& \leq \Big( \|U_1\|+\|U_2\|\Big)  \|U_1-U_2\|  \\
& \leq \frac 1 2 \|U_1-U_2\|  \,.
 \end{split}
\end{equation}
The main task is to derive a similar bound on the second term in \eqref{prop1}. 
We formulate this main result as a proposition and postpone its proof to the next section.

\begin{prop}\label{P.main}
For sufficiently small $\eps$ we have
\begin{equation}\label{main}
 \sup_{q>-1}  \frac{1}{1{+}q}\Big| \int_{-1}^q  \frac{(1+s)^2}{s^2} \Big( {\cal M}(f_1,f_1)(s) - {\cal M}(f_2,f_2)(s)\Big)\,ds \Big| \leq C \eps  \| U_1-U_2\|\,.
\end{equation}
\end{prop}
With Proposition \ref{P.main} the statement of the theorem follows.
\end{proof}

\subsection{Proof of  Proposition \ref{P.main}}

We first notice that it suffices to prove Proposition \ref{P.main} for $W(x,y)$ that satisfies \eqref{kernel2}-\eqref{kernel0} with $\eps=1$. The result then follows by
scaling. 
For the proof of Proposition \ref{P.main} we argue by contradiction. 
Suppose that \eqref{main} (with $\eps=1$) is not true. Then there exist sequences $\{W_n\}, \{f_{1,n}\}, \{f_{2,n}\}$ and $\{q_n\}$ such that, with $U_{i,n}$ denoting
the corresponding functions as above, 
\begin{equation}\label{assump1}
 \|U_{1,n} - U_{2,n}\| \to 0 \qquad \mbox{ as } n \to \infty
\end{equation}
and
\begin{equation}\label{assump2}
\frac{1}{q_n{+}1}\Big| \int_{-1}^{q_n}  \frac{(1+s)^2}{s^2} \Big( {\cal M}(f_{1,n},f_{1,n})(s) - {\cal M}(f_{2,n},f_{2,n})(s)\Big) \,ds\Big| \geq 1\,. 
\end{equation}

By our regularity assumption \eqref{kernel0} we can assume without loss of generality that there exists a function $W_*=W_*(x,y)$, satisfying \eqref{kernel1}-\eqref{kernel0} such that 
\begin{equation}\label{wnconvergence}
W_n \to W_* \qquad \mbox{ locally uniformly on } (0,\infty)^2\,.
\end{equation}

We now collect some a-priori estimates for solutions $f$.
\begin{lemma}
 Let $f$ be a solution to \eqref{eq1}. Then
\begin{align}
 \int_0^{\infty} f(x)\,dx & \leq 2 \,,\label{f1}\\
\int_0^1 \frac{f(x)}{x^{\alpha}}\,dx & \leq C_{\alpha}\,,\label{f2}\\
\Big| \int_0^{\infty}\big(1-e^{-qx}\big)f(x)\,dx \Big| &\leq \frac{2|q|}{1+q}\qquad \mbox{ for all } q>-1\,,\label{f3}\\
\int_R^{2R} e^x f(x)\,dx & \leq 4 R \qquad \mbox{ for   } R \geq \frac{1}{1-\log 2} \,, \label{f4}\\
\int_1^{\infty} \frac{e^x}{x^{3-\alpha}} f(x)\,dx &\leq C \,. \label{f5}
\end{align}

\end{lemma}
\begin{proof}
The first estimate \eqref{f1} and the third \eqref{f3} follow from \eqref{smallness}, the second \eqref{f2} has been proved in Lemma \ref{L.regularity}. We can now deduce
\eqref{f4} from \eqref{f3}. In fact, choosing $q<-\log 2$, we have
\[
\int_0^{\infty} \big( e^{-qx}-1\big) f(x)\,dx  = \Big | \int_0^{\infty} \big( 1- e^{-xq}\big) f(x)\,dx\Big| \geq \frac 1 2 \int_1^{\infty} e^{-qx} f(x)\,dx. 
\]
As a consequence we obtain
\[
\int_1^{\infty} e^{|q|x} f(x) \,dx \leq \frac{4}{1+q} \qquad \mbox{ for } q\in (-1,-\log 2)\,.
\]
Choosing now $1+q=\frac{1}{R}$ and $x \in (R,2R)$ estimate \eqref{f4} follows. 
Finally, estimate \eqref{f5} follows from \eqref{f4} via the usual dyadic argument, that is
\[
 \int_1^{\infty} \frac{e^x}{x^{3-\alpha}} f(x)\,dx \leq \sum_{n=0}^{\infty} \int_{2^n}^{2^{n{+}1}} \frac{e^x}{x^{3-\alpha}} f(x)\,dx \leq C \sum_{n=0}^{\infty} 
2^{n{+}1} 2^{-(3-\alpha)(n{+}1)} \leq C\,.
\]
\end{proof}
We now write
\begin{align*}
\frac{1}{q{+}1}&\int_{-1}^q \,ds \frac{(1+s)^2}{s^2} \Big( {\cal M}(f_1,f_1)(s) - {\cal M}(f_2,f_2)(s)\Big)\\
&= \int_0^{\infty} \int_0^{\infty} W(x,y) \big( f_1(x)+f_2(x)\big)\big(f_1(y)-f_2(y)\big) H(q,x,y)\,dx\,dy
\end{align*}
with
\begin{equation}\label{Hdef}
 H(q,x,y)= \frac{1}{1+q} \int_{-1}^q \frac{(1+s)^2}{s^2} \big(1-e^{-sx}\big) \big( 1-e^{-sy}\big)\,ds\,.
\end{equation}

\subsubsection{The case $q_n \to q^* \in (-1,\infty]$.}

Now assume that $q_n\to q^*\in (-1,\infty]$. In this case we can  use the following estimate for $H$. 
\begin{lemma}\label{L.Hestimate}
 For $q >-1+\frac{1}{L}$ we have
\[
 0 \leq H(q,x,y) \leq C_L \frac{\min(x,1)\min(y,1)}{1+(x+y)^3} e^{x+y}\,.
\]
\end{lemma}
\begin{proof}
If $x,y \leq 1$, the estimate is immediate by linearizing the function $1-e^{-sx}$. If $x,y \geq 1$, then the main contribution to the integral comes from the region
$s \sim -1$. In fact, if  $-1+1/L < q <-1/L$, then 
\begin{align*}
  H(q,x,y) &\leq C_L \int_{-1}^q (1+s)^2 e^{-s(x+y)}\,ds\\
& = C_L \frac{e^{x+y}}{(x+y)^3} \int_0^{(1+q)(x+y)} t^2 e^{-t}\,dt \\
& \leq C_L \frac{1}{1+(x+y)^3}e^{x+y}\,.
\end{align*}
In a neighborhood of $s=0$ we can again linearize, while for $s \geq \frac{1}{L}$ we just use the upper bound
$\big(1-e^{-sx}\big) \big( 1-e^{-sy}\big) \leq 1$. 

If e.g. $x \geq 1$ and $y \leq 1$, the result follows analogously.
\end{proof}

\begin{rem}
If $q_n \to q^* \in (-1,\infty)$, then it is obvious that $H(q_n,\cdot)$ converges locally uniformly in $(0,\infty)^2$ to $H(q^*,\cdot,\cdot)$.

If $q_n \to \infty$, then $H(q_n,\cdot)$ converges locally uniformly to $2$. 
\end{rem}

\bigskip
Then, if $q >-1+\frac{1}{L}$ and if $f$ is a solution to \eqref{eq1}, we have, using Lemma \ref{L.Hestimate}, that for large $R$
\begin{align*}
\int_0^{1/R} &\,dx \int_0^{\infty} dy W(x,y) f(x)f(y)  H(q,x,y)\\
&\leq C \int_0^{1/R}dx \int_0^{\infty} dy \Big( \Big(\frac{x}{y}\Big)^{\alpha} + \Big( \frac{y}{x}\Big)^{\alpha} \Big) f(x)f(y) \frac{x\min(y,1)}{1+(x+y)^3} e^{x+y}\\
& \leq C \int_0^{1/R} x^{1{-}\alpha} f(x)\,dx \, \int_0^{\infty} \Big( y^{-\alpha} + y^{\alpha}\Big) \frac{\min(y,1)e^y}{1+y^3} f(y)\,dy
\\ & \leq \frac{C}{R^{1{-}\alpha}}\,,
\end{align*}
where the last estimate follows from \eqref{f1} and \eqref{f5}.

Furthermore, using also \eqref{f4}, we arrive similarly at 
\begin{align*}
&\int_{R}^{\infty} \,dx \int_0^{\infty} dy W(x,y) f(x)f(y)  H(q,x,y)
\leq C \int_R^{\infty} \frac{x^{\alpha} e^x}{1+x^3} f(x)\,dx\\
& \leq C \sum_{n=0}^{\infty} \int_{R2^n}^{R2^{n{+}1}} \frac{e^x}{x^{3-\alpha}} f(x)\,dx \leq C \sum_{n=0}^{\infty}
\big( R 2^n\big)^{-3+\alpha} R 2^n \leq \frac{C}{R^{2-\alpha}} \sum_{n=0}^{\infty} 2^{n(-2+\alpha)} \leq \frac{C}{R^{2-\alpha}}\,.
\end{align*}
Hence, in order to arrive at a contradiction to \eqref{assump2}, it remains to show that for large but fixed $R$
\begin{equation}\label{prop3}
\int_{1/R}^R  \int_{1/R}^R  W_n(x,y) \big( f_{1,n}(x)+f_{2,n}(x)\big)\big(f_{1,n}(y)-f_{2,n}(y)\big) H(q_n,x,y)\,dx\,dy\to 0 \quad \mbox{ as } n \to \infty.
\end{equation}

Since $W_n$ and $H(q_n,\cdot,\cdot)$ converge locally uniformly to their respective limits and since assumption \eqref{assump1} in particular implies that
$f_{1,n}-f_{2,n} \to 0$ locally in the sense of measures, we find that
\[
 F_n(x):= \int_{1/R}^R W_n(x,y) H(q_n,x,y) \big(f_{1,n}(y)-f_{2,n}(y)\big) \,dy \to 0 \qquad \mbox{ as } n \to \infty 
\]
locally uniformly in $x$. Hence, we can derive \eqref{prop3} and we have proved a contradiction in case $q_n \to q^* \in (-1,\infty]$.

\subsubsection{The case $q_n \to -1$.}

This case is somewhat more difficult to treat. We introduce the rescaling
\begin{equation}\label{rescaling}
 X=(1+q_n)x \qquad \mbox{ and } g(X)=e^x f(x)\,.
\end{equation}
The integral (cf. \eqref{assump2}) for which we want to show that it converges to zero as $n \to \infty$ becomes
\begin{equation}\label{rescaledterm}
 \int_0^{\infty} \int_0^{\infty} W_n(X,Y) \big( g_{1,n}(X) + g_{2,n}(X)\big) \big( g_{1,n}(Y) - g_{2,n}(Y)\big) \tilde H(q_n,X,Y)\,dX\,dY
\end{equation}
with
\begin{equation}\label{Htildedef}
 \tilde H(q_n,X,Y)= \frac{e^{-(X+Y)/(1+q_n)}}{(1+q_n)^2} H\Big( q_n, \frac{X}{1+q_n}, \frac{Y}{1+q_n}\Big)\,.
\end{equation}

We are going to derive a-priori estimates for $g$ and $\tilde H$.

\begin{lemma}\label{L.gestimates}
 We have for any solution $f$ of \eqref{eq1} and $g$ defined as in \eqref{rescaling} that
\begin{align}
 \int_0^{\infty} e^{-\frac{X}{1+q_n}} g(X)\,dX & \leq C(1+q_n) \,, \label{g1}\\
\int_0^{2(1+q_n)} \frac{g(X)}{X^{\alpha}} \,dX &\leq C_{\alpha} (1+q_n)^{1-\alpha}\,, \label{g2}\\
\int_R^{2R} g(X)\,dX & \leq CR \qquad \mbox{ for all } R \geq 2(1+q_n)\,\label{g3}\\
\int_0^{\infty} \Big( Y^{\alpha} + Y^{-\alpha}\Big) \frac{\min(Y/(1+q_n),1)}{1+Y^3}g(Y)\,dY & \leq C\,.\label{g4}
\end{align}
\end{lemma}
\begin{proof}
The first estimate \eqref{g1} follows from \eqref{f1} and the definitions in \eqref{rescaling}, while estimate \eqref{g2} is a consequence of \eqref{f2}.

To establish \eqref{g3} we deduce from \eqref{f3} that for $q<0$ we have
\begin{align*}
\int_0^{\infty} e^{-X\frac{1+q}{1+q_n}}g(X)\,dX & \leq 2|q| \frac{1+q_n}{1+q} + \int_0^{\infty} e^{-\frac{X}{1+q_n}}g(X)\,dX\\
& \leq   2|q| \frac{1+q_n}{1+q} + C(1+q_n)\,.
\end{align*}
We choose $R=\frac{1+q_n}{1+q}$ for any $q \in (-1,-1/2)$ to infer \eqref{g3}. 

Finally, we use the usual dyadic argument and \eqref{g3} to estimate
\begin{align*}
 \int_1^{\infty} Y^{\alpha-3} g(Y)\,dY & \leq \sum_{k=1}^{\infty} \int_{2^k}^{2^{k{+}1}} Y^{\alpha-3} g(Y)\,dY\\
& \leq C\sum_{k=1}^{\infty} 2^{k(\alpha-3)} 2^k \leq C \sum_{k=1}^{\infty} 2^{-k(2-\alpha)} \leq C
\end{align*}
as well as 
\[
 \int_{2(1+q_n)}^1 Y^{-\alpha} g(Y)\,dY\leq C \sum_{k=0}^{2^{-k} \geq 2(1+q_n)} \int_{2^{-(k{+}1)}}^{2^{-k}} Y^{-\alpha} g(Y)\,dY
\leq C \sum_{k=0}^{2^{-k} \geq 2(1+q_n)} 2^{k\alpha} 2^{-k} \leq C
\]
which together with \eqref{g2} gives \eqref{g4}.
\end{proof}

\begin{lemma}\label{L.Htildeestimates}
 \[
  \tilde H(q_n,X,Y) \leq C \frac{\min(X/(1+q_n),1)\min(Y/(1+q_n),1)}{1+(X+Y)^3}\,.
 \]
\end{lemma}
\begin{proof}
Using the definitions of $\tilde H$, the estimate follows exactly as in the proof of Lemma \ref{L.Hestimate}.
\end{proof}

With these estimates we can control the regions near zero and infinity. Indeed, using \eqref{g2},  \eqref{g4}
and Lemma \ref{L.Htildeestimates}, we obtain
\begin{align*}
& \int_0^{1/R}dX \int_0^{\infty} dY W_n(X,Y) \big( g_{1,n}(X) + g_{2,n}(X)\big) \big( g_{1,n}(Y) - g_{2,n}(Y)\big) \tilde H(q_n,X,Y)\,dX\,dY\\
& \leq C \int_0^{1/R}dX\int_0^{\infty} dY \Big( \Big(\frac{X}{Y}\Big)^{\alpha} + \Big( \frac{Y}{X}\Big)^{\alpha}\Big) g_{1,n}(X)g_{2,n}(Y) \tilde H(q_n,X,Y)\\
& \leq C \sum_{j=0}^{\stackrel{2^{-j}\geq}{ R(1+q_n)}} \int_{\frac{2^{-(j+1)}}{R}}^{\frac{2^{-j}}{R}} dX \int_0^{\infty} dY \Big( \big(2^{j}RY\big)^{-\alpha}
+ \big(2^{j}R Y\big)^{\alpha}\Big) g_{1,n}(X)g_{2.n}(Y) \frac{\min(\frac{Y}{1+q_n},1)}{1+Y^3}\\
& \qquad + \int_0^{1+q_n}dX \int_0^{\infty} \Big( \Big(\frac{X}{Y}\Big)^{\alpha} + \Big( \frac{Y}{X}\Big)^{\alpha}\Big) \frac{X}{1+q_n} g_{1,n}(X) g_{2.n}(Y) 
\frac{\min(\frac{Y}{1+q_n},1)}{1+Y^3}\\
& \leq C \sum_{j=0}^{2^{-j} \geq R (1+q_n)} \big( 2^jR\big)^{-(1{+}\alpha)} + \big(2^{j} R\big)^{\alpha{-}1} + C (1+q_n)^{1-\alpha}\\
& \leq C \Big( R^{\alpha{-}1} + (1+q_n)^{1{-}\alpha}\Big)\,.
\end{align*}

Second, we estimate
\begin{align*}
 &\int_R^{\infty} dX \int_{1/R}^{\infty} dY W_n(X,Y) \big( g_{1,n}(X) + g_{2,n}(X)\big) \big( g_{1,n}(Y) - g_{2,n}(Y)\big) \tilde H(q_n,X,Y)\,dX\,dY\\
& \leq C \int_R^{\infty} dX \int_{1/R} ^{\infty} dY  \Big( \Big(\frac{X}{Y}\Big)^{\alpha} + \Big( \frac{Y}{X}\Big)^{\alpha}\Big)
 g_{1,n}(X)g_{2,n}(Y) \frac{\min(\frac{Y}{1+q_n},1)}{1+(X+Y)^3} \\
&=:  \int_R^{\infty} \int_{1/R}^R dY \cdots + C \int_R^{\infty} dX \int_R^{\infty} dY ... =: (I)+(II)
\end{align*}
Using \eqref{g3} we obtain
\begin{align*}
 (I)& \leq C 
\int_R^{\infty} dX \int_{1/R}^R \Big( \Big(\frac{X}{Y}\Big)^{\alpha} + \Big( \frac{Y}{X}\Big)^{\alpha}\Big) \frac{g_{1,n}(X)g_{2,n}(Y)}{X^3}\\
& \leq C \sum_{j=0}^{\infty} \sum_{k=0}^{2^{-k} \geq 1} \int_{R2^j}^{R2^{j{+}1}}dX \int_{R 2^{-(k{+}1)}}^{R2^{-k}} dY \Big( 2^{\alpha(j{+}k)} + 2^{-\alpha(j{+}k)}\Big)
\frac{g_{1,n}(X)g_{2,n}(Y)}{(R2^j)^3}\\
& \leq \frac{C}{R} \sum_{j=0}^{\infty} \sum_{k=0}^{\infty}\Big( 2^{\alpha(j{+}k)} + 2^{-\alpha(j{+}k)}\Big) 2^{-2j-k}\\
& \leq \frac{C}{R} \sum_{j=0}^{\infty} \sum_{k=0}^{\infty}
\Big( 2^{j(\alpha-2)} 2^{k(\alpha-1)} + 2^{-j(\alpha{+}2)} 2^{-k(1{+}\alpha)}\Big)\leq \frac{C}{R}\,.
\end{align*}
Furthermore, 
\begin{align*}
 (II)&\leq \sum_{j=0}^{\infty} \sum_{k=0}^{\infty} \int_{R2^j}^{R2^{j{+}1}} dX \int_{R2^k}^{R2^{k{+}1}}dY \Big( 2^{\alpha(j{-}k)} + 2^{\alpha(k{-}j)}\Big) \frac{g_{1,n}(X)
g_{2,n}(Y)}{R^3 (2^j+2^k)^3}\\
& \leq \frac{C}{R} \sum_{j,k=0}^{\infty} \Big( 2^{\alpha(j{-}k)} + 2^{\alpha(k{-}j)}\Big)\frac{2^{j+k}}{(2^j+2^k)^3}\\
& \leq \frac{C}{R} \int_1^{\infty} d\xi \int_1^{\infty}d\eta \frac{\big( \frac{\xi}{\eta}\big)^{\alpha} + \big( \frac{\eta}{\xi}\big)^{\alpha}}{(\xi+\eta)^3} \\
& \leq \frac{C}{R} \int_1^{\infty} \frac{dr}{r^2} \int_0^{\pi/2} d\theta
\big( \tan \theta^{\alpha} + \mbox{cotan} \theta^{\alpha}\big)
 \leq \frac{C}{R}.
\end{align*}
Thus it remains to show that
\begin{equation}\label{final}
 \int_{1/R}^RdX \int_{1/R}^R dY W_n(X,Y) \big (g_{1,n}(X) + g_{2,n}(X)\big) \big( g_{1,n}(Y) - g_{2,n}(Y)\big) \tilde H(q_n,X,Y) \to 0
\end{equation} 
as $ n \to \infty$.

\begin{lemma}\label{L.Htildeconverge}
As $q_n\to -1$ we have
\[
 \tilde H(q_n,X,Y) \to \frac{1}{(X+Y)^3} \int_0^{X+Y} \xi^2 e^{-\xi}\,d\xi \quad \mbox{ locally uniformly in } (X,Y) \in (0,\infty)^2\,.
\]
\end{lemma}
\begin{proof}
We have
\begin{align*}
 \tilde H(q_n,X,Y) &= \frac{1}{(1+q_n)^3} \int_{-1}^{q_n} \frac{(1+s)^2}{s^2} \Big( 1 - e^{-s\frac{X}{1+q_n}}\Big) \Big( 1 - e^{-s\frac{Y}{1+q_n}}\Big)e^{-\frac{X+Y}{1+q_n}}\,ds\\
& = \frac{1}{(1+q_n)^3} \int_{-1}^{q_n}  \frac{(1+s)^2}{s^2} e^{- \frac{1+s}{1+q_n} (X+Y)}\,dx \\
&+ \frac{1}{(1+q_n)^3} \int_{-1}^{q_n}  \frac{(1+s)^2}{s^2}
\Big( 1 - e^{-\frac{sX}{1+q_n}} - e^{-\frac{sY}{1+q_n}}\Big) e^{-\frac{X+Y}{1+q_n}}\,ds\,.
\end{align*}
 We can estimate the second term on the right hand side  as
\begin{align*}
\Big| \frac{1}{(1+q_n)^3} &\int_{-1}^{q_n}  \frac{(1+s)^2}{s^2}
\Big( 1 - e^{-\frac{sX}{1+q_n}} - e^{-\frac{sY}{1+q_n}}\Big) e^{-\frac{X+Y}{1+q_n}}\,ds\Big|\\
&\leq \frac{1}{(1+q_n)^3} \int_{-1}^{q_n}  \frac{(1+s)^2}{s^2}\Big(e^{-\frac{X+Y}{1+q_n}} + e^{-X\frac{1+s}{1+q_n}} e^{-\frac{Y}{1+q_n}} + e^{-Y\frac{1+s}{1+q_n}} e^{-\frac{X}{1+q_n}}\Big)\,ds
\end{align*}
and since e.g. $e^{-X\frac{1+s}{1+q_n}} \leq C$ and $ e^{-\frac{Y}{1+q_n}} \to 0$ as $q_n \to -1$, we can deduce that the whole term converges to zero as $n \to \infty$.

On the other hand,
\[
 \frac{1}{(1+q_n)^3} \int_{-1}^{q_n}  \frac{(1+s)^2}{s^2} e^{- \frac{1+s}{1+q_n} (X+Y)}\,dx\\
= o(1) + \int_0^1 \xi^2 e^{-\xi(X+Y)}\,d\xi
\]
and the result follows after another rescaling of the integral on the right hand side.
\end{proof}

\begin{lemma}\label{L.gconverge}
 If $\|f_{1,n} - f_{2,n}\| \to 0 $ as $n \to \infty$, then $\int_0^{\infty} e^{-\theta X} \big (g_{1,n}(X)-g_{2.n}(X)\big) dX \to 0$ for all $\theta>0$. Hence 
$g_{1,n}-g_{2,n} \to 0$ weakly in the sense of measures.
\end{lemma}
\begin{proof}
 By the definitions we have 
\begin{align*}
 0& \leftarrow \sup_{q>-1}\frac{1+q}{|q|} \Big| \int_0^{\infty} \big(1-e^{-qx}\big) (f_{1,n}(x)-f_{2,n}(x)) \,dx\Big| \\
& = \sup_{q>-1} \frac{1+q}{|q|(1+q_n)} \Big| \int_0^{\infty} \big( 1-e^{- \frac{qX}{1+q_n}}\big) e^{-\frac{X}{1+q_n}} \big (g_{1,n}(X)-g_{2,n}(X)\big) dX \Big|\,.
\end{align*}
Given $\theta>0$ we define $\tilde q_n \to -1$ such that $\theta(1+q_n)=1+\tilde q_n$. Then,  using \eqref{g1}, \eqref{g2} and \eqref{g3}, we find
\[
 \Big| \int_0^{\infty} e^{- \theta X} \big (g_{1,n}(X)-g_{2,n}(X)\big) dX \Big|
 \leq o(1) + C \Big| \int_0^{\infty} e^{-\frac{X}{1+q_n}} \big (g_{1,n}(X)-g_{2,n}(X)\big) dX \Big|
 \to 0 
\]
as $n\to \infty$.
Since,  the left hand side is just  the Laplace transform of $g_{1,n}-g_{2,n}$, this proves the statement of the Lemma.
\end{proof}
With Lemmas \ref{L.Htildeconverge} and  \ref{L.gconverge} 
 we can deduce that \eqref{final} holds. This gives a contradiction to \eqref{assump2} and finishes the proof
of Proposition \ref{P.main}.

\bigskip
{\bf Acknowledgment.} The authors acknowledge support through the CRC 1060 {\it The mathematics of emergent effects } at the University of Bonn, that is funded through the
German Science Foundation (DFG).

\section{Appendix: Proof of Lemma \ref{L.ublargex}}

\begin{proof}
Dividing \eqref{eq1c} by $x$, integrating and changing the order of integration on the right hand side we derive
as a first a priori estimate that
\begin{align}
 \int_0^{\infty} x f(x)\,dx & = \int_0^{\infty} \int_0^{\infty} K(y,z) f(y) f(z) y \log \Big( \frac{y+z}{y}\Big)\nonumber\\
& \geq C \int_0^{\infty} \,dz \int_0^z \,dy f(y) f(z) K(y,z) y \log \Big( \frac{y+z}{y}\Big)\label{ublargex1}\\
& \geq C \int_0^{\infty} \,dz \int_0^z \,dy f(y) f(z) K(y,z)y\,.\nonumber
\end{align}

Next, we denote for $\gamma \geq 1$
\begin{equation}
\label{ublargex2}
 M(\gamma):=\int_0^{\infty} x^{\gamma} f(x)\,dx \,.
\end{equation}
Our goal is to show inductively that $M(\gamma) \leq \gamma^{\gamma} e^{A\gamma}$
for some (large) constant $A$.

To that aim we first  multiply \eqref{eq1b}  by  $x^{\gamma-2}$ with some $\gamma>1$ and after integrating we obtain 
\[
(\gamma{-}1) M(\gamma) = \tfrac 1 2 \int_0^{\infty}\int_0^{\infty} K(x,y) f(x) f(y)\big( (x+y)^{\gamma} - x^{\gamma} - y^{\gamma} \big)\,.
\]
By symmetry we also find
\begin{align*}
 M(\gamma)&= \frac{1}{\gamma{-}1} \int_0^{\infty} \,dx \int_0^x \,dy K(x,y) f(x) f(y) \Big( \big(x+y\big)^{\gamma} - x^{\gamma}\Big)\\
& = \int_0^{\infty} \,dx \int_0^{x/\gamma} \,dy \cdots + \int_0^{\infty} \,dx \int_{x/\gamma}^x \,dy \cdots\,.
\end{align*}
Due to \eqref{ubsmallx} we have
\begin{align}
 \int_0^1\,dx \int_0^x \,dy K(x,y) f(x) f(y) \Big( \big(x+y\big)^{\gamma} - x^{\gamma}\Big)&
\leq C \int_0^1 \,dx \int_0^x \,dy K(x,y) f(x) f(y)  x^{\gamma}\nonumber \\
& \leq C\int_0^1 x^{\alpha +\gamma}f(x)\int_0^x y^{1{-}\alpha} f(y)\,dy\,dx \label{ublargex4}\\
& \leq C \,. \nonumber
\end{align}
Using \eqref{ubsmallx} and  $(x+y)^{\gamma} - x^{\gamma} \leq c x^{\gamma-1}y$ for $ y \leq \frac{x}{\gamma}$, we find that
\begin{equation}\label{ublargex5}
\begin{split}
 \int_1^{\infty}  \int_0^{x/\gamma} K(x,y)  y x^{\gamma{-}1} f(x)f(y)\,dy\,dx&\leq \int_1^{\infty} \int_0^{x/\gamma} x^{\gamma+\alpha-1} y^{1-\alpha} f(x)f(y)\,dy\,dx 
\\ &\leq
 C M(\gamma+\alpha -1)\,,
\end{split}
\end{equation}
so that for the sum of both terms we can prove by induction that it is smaller than $1/2 \gamma^{\gamma} e^{A\gamma}$. 
It remains to estimate
\begin{align*}
\frac{C}{\gamma-1}&\int_1^{\infty} \,dx \int_{x/\gamma}^x\,dy   K(x,y) f(x) f(y) \big(x+y\big)^{\gamma}\\
& \leq \frac{C}{\gamma} \int_1^{\infty} \,dx \int_{x/\gamma}^x \,dy f(x) f(y) \big(x+y\big)^{\gamma} \Big( \frac{x}{y}\Big)^{\alpha}=:(*)\,.
\end{align*}
In the following $\{\zeta_n\} \subset (0,1]$ will be a decreasing sequence of numbers that will be specified later.
Then we define a corresponding sequence of numbers $\kappa_n$ such that given a sequence $\{\theta_n\} \subset (0,1)$, also to be
specified later, we have
\begin{equation}\label{kappandef}
 \big(x+y\big)^{\gamma}  \leq \kappa_n^{\gamma} x^{\gamma(1-\theta_n)} y^{\gamma \theta_n} \qquad \mbox{ for } \frac{y}{x} \in [\zeta_{n+1},\zeta_n]\,.
\end{equation}
Equivalently we have
\begin{equation}\label{kappadef1}
 \kappa_n = \max_{\zeta \in [\zeta_{n{+}1},\zeta_n]} \Big( \frac{1+\zeta}{\zeta^{\theta_n}}\Big)\,.
\end{equation}
With these definitions we have
\[
 (*) \leq \frac{C}{\gamma} \sum_{n=0}^{n_0(\gamma)} \kappa_n^{\gamma} \zeta_{n{+}1}^{-\alpha} M(\gamma(1-\theta_n)) M(\gamma \theta_n)\,,
\]
where $n_0(\gamma)$ is such that $\zeta_{n_0(\gamma)} = \frac{1}{\gamma}$. 

We  choose $\theta_n$ such that for $\psi_{\theta_n}(\zeta):= \log(1+\zeta)-\theta_n \log \zeta$ we have
\[
\min_{\zeta \in [\zeta_{n+1},\zeta_n]} \psi_{\theta_n}(\zeta) = \log(1+\zeta_n) - \theta_n \log(\zeta_n)\,.
\]
This is equivalent to 
\begin{equation}\label{thetandef}
\theta_n = \frac{\zeta_n}{1+\zeta_n}\,.
\end{equation}
We want to prove now by induction over $\gamma$ that $(*) \leq \frac{1}{2} \gamma^{\gamma} e^{A\gamma}$. Inserting the corresponding hypothesis, this reduces to showing that
\[
 \frac{C}{\gamma} \sum_{n=0}^{n_0} \zeta_{n{+}1}^{-\alpha} \exp \Big( \gamma \big( \max_{\zeta \in [\zeta_{n+1},\zeta_n]} \psi_{\theta_n}(\zeta)
+ \theta_n \log \theta_n + (1-\theta_n) \log (1-\theta_n)\big) \Big)\leq \frac 1 2\,.
\]
By our definition \eqref{thetandef} we have
\begin{align*}
 &\max_{\zeta \in [\zeta_{n+1},\zeta_n]} \psi_{\theta_n}(\zeta)
+ \theta_n \log \theta_n + (1-\theta_n) \log (1-\theta_n) \\
& = \min_{\zeta \in [\zeta_{n+1},\zeta_n]}\psi_{\theta_n}(\zeta)  + (\max-\min)_{\zeta \in [\zeta_{n+1},\zeta_n]}  \psi_{\theta_n}(\zeta)
+ \theta_n \log \theta_n + (1-\theta_n) \log (1-\theta_n) \\
& = (\max-\min)_{\zeta \in [\zeta_{n+1},\zeta_n]}  \psi_{\theta_n}(\zeta)\,.
\end{align*}
Thus we need to investigate

\begin{align*}
 (\max-\min)_{\zeta \in [\zeta_{n+1},\zeta_n]}  \psi_{\theta_n}(\zeta) &= \psi_{\theta_n}(\zeta_{n{+}1}) - \psi_{\theta_n}(\zeta_n)\\
& = \log \Big( \frac{1+\zeta_{n{+}1}}{1+\zeta_n} \Big) - \theta_n \log \Big( \frac{\zeta_{n{+}1}}{\zeta_n} \Big)\\
& = \log \Big( 1 + \frac{\zeta_{n{+}1}-\zeta_n}{1+\zeta_n}\Big) - \frac{\zeta_n}{1+\zeta_n} \log \Big( 1 + \frac{\zeta_{n{+}1}-\zeta_n}{\zeta_n}\Big)\,.
\end{align*}
Expanding the nonlinear terms we find
\begin{equation*}
V:=(\max-\min)_{\zeta \in [\zeta_{n+1},\zeta_n]}  \psi_{\theta_n}(\zeta) \leq C \Big( |\zeta_{n{+}1}-\zeta_n|^2 + \frac{(\zeta_{n{+}1}-\zeta_n\big)^2}{
\zeta_n}\Big)\,.
\end{equation*}
We now split $\{1,2,\cdots,n_0\}$ into a finite number of sets $\{1,2,\cdots,N_1\}$, $\{N_1+1, \cdots, N_2\}$, $\cdots$, $\{N_{k-1}+1, \cdots,N_k=n_0\}$ in the following way.

We first define 
\[
 \zeta_0 =1 \,, \quad \eta_0= 1+ \frac{1}{\sqrt{\gamma}}\,, \quad \zeta_n = \eta_0^{-n} \zeta_0\,, \quad \mbox{ for all } n \leq N_1
\]
where $N_1$ is such that $ \zeta_{n} \geq \frac{1}{\sqrt{\gamma}}$, that is we can choose $N_1 \sim \sqrt{\gamma} \log \gamma$.
With these definitions we find 
\[
 \Big| \frac{\zeta_{n{+}1} - \zeta_n}{\zeta_n}\Big| \leq \frac{C}{\sqrt{\gamma}} \qquad \mbox{ for all } 1\leq n\leq N_1
\]
and thus
\[
 \frac{1}{\gamma} \sum_{n=0}^{N_1} \zeta_{n{+}1}^{-\alpha} \exp \big( \gamma V\big) \leq \frac{C N_1}{\gamma} \gamma^{\alpha/2}  
\sim \gamma^{(\alpha-1)/2} \log \gamma \to 0 \quad \mbox{ as } \gamma \to \infty\,.
\]
Next, for $n \in (N_1,N_2]$ we define
\[
 \zeta_n=\eta_1^{-(n-N_1)} \zeta_{N_1}\,, \qquad \eta_1=1+\frac{1}{\gamma^{1/4}}\,.
\]
Then $\zeta_n \leq \frac{1}{\sqrt{\gamma}}$, such that $|\zeta_{n{+}1}-\zeta_n|^2 \leq \frac{C}{\gamma}$ and 
\[
\gamma V \leq C \gamma \Big( \zeta_n |\eta_1-1|^2 + \frac{1}{\gamma}\Big) \leq C\,.
\]
We need to determine $N_2$ such that
\[
(N_2-N_1) \frac{1}{\zeta_{N_2}^{\alpha} \gamma} \to 0 \qquad \mbox{ as } \gamma \to \infty\,.
\]
Making the ansatz $\zeta_{N_2} = \gamma^{-\sigma}$, that is $N_2 \sim \gamma^{1/4} \log \gamma$, this implies that we need
\[
 \frac{\gamma^{1/4} \gamma^{\alpha \sigma} \log \gamma}{\gamma} \ll 1\,
\]
and this needs $\alpha \sigma < \frac 3 4$. Hence $\sigma:= \min (1,\frac{3}{4\alpha})$. If $\sigma=1$ we are done, otherwise we need to iterate.

Thus, let us assume that $\sigma_1=\frac{3}{4\alpha} <1$.

We define for $k \geq 2$ the sequence $\sigma_{k+1} = \frac{1}{2\alpha} (1+\sigma_k)$.
Then, we define $\eta_{k+1}= 1+\frac{1}{\gamma^{(1-\sigma_{k+1})/2}}$ and $\zeta_n = \eta_{k+1}^{-(n-N_{k+1})} \zeta_{N_{k+1}}$ for $n \in (N_{k+1},N_{k+2}]$ with 
$\zeta_{N_{k+1}}=\gamma^{-\sigma_k}$. Then $N_{k+1}-N_k = \gamma^{(1-\sigma_k)/2} \log \gamma $ and we find that our sum is controlled by
$C \gamma^{(1-\sigma_k)/2 -1 + \alpha \sigma_{k+1}} \ll 1$ by our definition of $\sigma_{k+1}$. Since $\alpha<1$ we find after a finite number of steps that
$\sigma_k \geq 1$, and then we can stop the iteration.

It remains to show that \eqref{ublargex2} implies the pointwise estimate for $f$. Indeed, \eqref{ublargex2} implies for $R>0$ that
\[
 R^{\gamma} \int_{R}^{2R} f(x)\,dx \leq \int_{R}^{2R} x^{\gamma} f(x)\,dx \leq \gamma^{\gamma} e^{A\gamma}
\]
and thus
\[
 \int_R^{2R} f(x)\,dx \leq \exp\Big( \gamma (\log(\gamma)+ \log(R)) + A \gamma\Big)\,.
\]
The minimum of $\psi(\gamma):= \gamma (\log(\gamma) +\log(R)) + A \gamma$ is obtained for $\gamma = e^{-(A+1)} R$ 
and thus
\[
 \int_R^{2R} f(x)\,dx \leq \exp\Big( - e^{-(A+1)} R\Big)
\]
and obviously it follows  that there exists (another) $A>0$ such that
\begin{equation}\label{ublargex20}
 \int_R^{2R} f(x)\,dx \leq \exp\Big( - A R\Big)\,.
\end{equation}
Equation \eqref{eq1c} implies
\[
 x^2 f(x) = \int_0^x\,dy \int_{x{-}y}^x\,dz  K(y,z) y f(y)f(z) + \int_0^x\,dy \int_x^{\infty}\,dz  K(y,z) y f(y)f(z)\,.
\]
Now we use \eqref{negativemoment} and \eqref{ublargex20} to conclude
\begin{align*}
 \int_0^x\,dy \int_x^{\infty}\,dz  K(y,z) y f(y)f(z)& \leq C \int_0^x y^{1{-}\alpha} f(y)\,dy \int_x^{\infty} z^{\alpha} f(z)\,dz\\
& \leq C \sum_{n=0}^{\infty} \int_{2^nx}^{x^{n{+}1}x} z^{\alpha} f(z)\,dz\\
& \leq C \sum_{n=0}^{\infty} \big( 2^nx\big)^{\alpha} \exp \big(- A 2^n x\big)\\
& \leq C \exp \Big(-\frac{A}{2} x\Big)\,.
\end{align*}
Similarly, we can estimate by symmetry
\begin{align*}
 \int_0^x\,dy \int_{x{-}y}^x\,dz  K(y,z) y f(y)f(z) & \leq C \int_{x/2}^x \,dy \int_{x{-}y}^x \,dz K(y,z)y f(y) f(z)\\
& \leq C \int_{x/2}^x  y^{1{+}\alpha} f(y)\,dy \int_0^x z^{-\alpha} f(z)\,dz\\
& \leq C\int_{x/2}^x y^{1{+}\alpha}  f(y)\,dy \leq C \exp \Big( - \frac{A}{4} x\
\Big)\,,
\end{align*}
which implies the desired estimate.
\end{proof}

{\small
\bibliographystyle{alpha}%
\bibliography{../coagulation}%
}

\end{document}